\documentclass[12 pt]{article}
\makeatletter
\@addtoreset{section}{part}
\usepackage{amsmath,amssymb,amsthm,setspace,mathrsfs,pifont,amsfonts}
\usepackage{hyperref, hyperref, hyperref, hyperref}
\usepackage{graphicx}
\usepackage[font=small, labelfont=bf]{caption}
\usepackage{geometry}
\usepackage{tikz}
\usetikzlibrary{matrix, graphs, graphs.standard}

\newtheorem{theorem}{Theorem}[section]
\newtheorem{corollary}{Corollary}[theorem]

\newtheorem{lemma}[theorem]{Lemma}
\newtheorem{defn}{Definition}[section]
\theoremstyle{remark}

\def\multichoose#1#2{\ensuremath{\left(\kern-.3em\left(\genfrac{}{}{0pt}{}{#1}{#2}\right)\kern-.3em\right)}}

\title{Classification of Minimal Separating Sets of Low Genus Surfaces\footnote{This work was made possible in part thanks to Research Computing at Portland State University and its HPC resources acquired through NSF grant numbers 2019216 and 1624776.
}}
\author{Christopher N. Aagaard, J.J.P. Veerman}

\begin{document}

\maketitle
\section{Introduction}
Given a metric space $(X,d)$ and two distinct points $p,q \in X$, the \textit{mediatrix of $p$ and $q$}, which we'll denote $L_{pq}$, is
\[L_{pq} := \{x \in X | d(p,x) = d(q,x)\}.\]
In other work it is also known as an equidistant set, bisector, ambiguous locus, midset, medial axis, or central set (\cite{fox},\cite{ponce}).  In Brillouin spaces, a large class of metric spaces which includes all Riemannian manifolds, it was shown in \cite{veerman} that all mediatrices have the following minimal separating property:
\begin{defn} A subset $M \subset X$ of a topological space $X$ is called a separating set if $X\setminus M$ is disconnected.  $M$ is \textit{minimal separating} in $X$ if $M$ is a separating set for $X$ and for
any $M'\subsetneq M$, $X\setminus M'$ is connected.\end{defn}

Additionally, it is shown in \cite{bernhard} that when $X$ is a compact Riemann surface then $L_{p,q}$ is a finite topological graph.  Then, in \cite{fox}, this result is generalized to show that for any compact 2-dimensional
Alexandrov space $X$ and any $\{p,q\}\subset X$, $L_{p,q}$ is a finite topological graph.  That leads us to investigate which topological graphs can be realized as minimal
separating sets in surface of genus $g$.  (From here on unless otherwise specified, surface is taken to mean a connected compact oriented surface without boundary).

In this paper, we will provide a classification of precisely when a graph embedded in a surface is a minimal separating set and then use that result to classify the topological graphs which can
be realized as minimal separating sets in a surface of genus $g$.  We then use this classification to write a computer program which determines the number of homeomorphism classes of minimal separating sets in surfaces of genus at most 5.  As seen in Table \ref{tab:results}, the number of such homeomorphism classes for genus 0, 1, 2, 3, 4, and 5 are respectively 1, 5, 26, 217, 3555, and 325734.

In section 2 we define useful terminology, and establish a simple but useful lemma.  Then in section 3, we prove a lemma relating minimal separating embeddings with more commonly studied cellular embeddings
 and then use this to establish a relationship between general (potentially disconnected) minimal separating sets in a surface and connected minimal separating sets, which are more easily compute.
 Section 4 introduces the language of topological and combinatorial maps to describe embeddings of connected graphs, and then more general objects called hypermaps, which we can use to speed up our
 computation. Section 5
describes an algorithm to find all homeomorphism classes of graphs which embed as minimal separating sets in a surface and has computational results from using that algorithm to find all graphs that embed as minimal separating sets in surfaces of genus at most 5.

Reframing the problem in terms of combinatorial maps presents an interesting connection to existing work in map enumeration, particularly the enumeration of unrooted maps discussed in \cite{liskovets}, \cite{mednykh_enumeration_2006}, and \cite{mednykh_enumeration_2010} which we discuss in the concluding remarks.  We require an actual \textit{list} of maps in order to determine numbers of
homeomorphism classes of minimal separating sets for a given genus, which makes the purely enumerative methods in those papers insufficient here, but the particular set of combinatorial maps we consider
presents some difficulties in applying their methods and provides a set to consider for map enumeration problems.

\section{Background and Definitions}
Throughout this paper we will use the notation $X_g$ to denote the compact surface (without boundary) of genus $g$.

As observed in \cite{bernhard}, if a graph $G$ embeds as a minimal separating set in $X_g$, then it also embeds as a minimal separating set in any surface of genus greater than $g$.
Such an embedding can be constructed by taking a minimal separating embedding $\eta$ of $G$ in $S$ and adding handles within a component of $S\setminus \eta(G)$ until the resulting surface has the
desired genus.

Given this observation, we define the following three sets:
\begin{defn}\label{def:graphsets}\;\\
i) $\mathcal{M}_g$ is the set of all graphs which can be realized as minimal separating sets in $X_g$.\\
ii) $\mathcal{L}_g$ is the set of all graphs which can be realized as minimal separating sets in $X_g$, but not in $X_{g-1}$. We say they have
least separating genus $g$.\\
iii) $\mathcal{C}_g$ is the set of all connected graphs with least separating genus $g$.\footnote{$\mathcal{M}$ for \textit{minimal separating}, $\mathcal{L}$ for \textit{least separating genus}, and $\mathcal{C}$ for \textit{connected}}
\end{defn}

We restate the above observation from $\cite{bernhard}$ as $\mathcal{L}_g = \mathcal{M}_g \setminus \mathcal{M}_{g-1}$.  The relationship between $\mathcal{L}_g$ and $\mathcal{C}_g$ is more involved and we
return to it later.  For now we merely remark that it is often easier to deal with connected graphs so it will be convenient to focus on the connected case.

Much of what follows will consider graphs as combinatorial objects so we should mention a key distinction between topological graphs and combinatorial graphs which is the distinction between graph
homeomorphism and graph isomorphism.

Any topological graph $G$ with at least two vertices and a vertex $v$ of degree 2 is homeomorphic to another topological graph $\hat{G}$ where $v$ and its 2
incident edges are replaced by a single edge (by convention we require a graph has at least one vertex, so we cannot replace a vertex if it is the only vertex).  So, given a connected topological graph $G$ with no component homeomorphic to a circle, we can find a unique topological graph $\hat{G}$, homeomorphic to $G$, and with no degree 2 vertices,
which we can construct by replacing all degree 2 vertices by single edges.  Also, given any $G$ we can construct an infinite family of graphs which are homeomorphic to $G$ by repeatedly
replacing single edges with 2 edges subdivided by a vertex.  This means that every mediatrix is homeomorphic to infinitely many topological graphs, so, when we talk about topological graphs,
 we always implicitly talk about equivalence classes of graphs up to homeomorphism, and identify each equivalence class with its unique representative with no degree 2 vertices (except in the case of circles, which we identify with the graph composed of a single vertex with a loop).

For the circle, we can see immediately that this embeds as a minimal separating set any surface, and in \cite{bernhard} it's shown that any minimal separating set in a surface of genus 0 is homeomorphic to such a graph.  Thus $\mathcal{C}_0 = \mathcal{L}_0 = \mathcal{M}_0$ consists of the graph
with one vertex and one loop, and for all $g\ge 1$, we identify elements of $\mathcal{C}_g$ with their unique representative with no degree 2 vertices.


Whether a graph can be realized as a minimal separating set for a surface $X$ is a question about \textit{embeddings} of a graph, so we introduce some terminology to discuss graph embeddings.

\begin{defn}
An embedding of a (topological) graph $G$ into a surface $X$ is a continuous function $\eta: G\to X$ such that $\eta$ is a homeomorphism onto its image.  We say an embedding is a cellular embedding or
a topological map\footnote{The name `map' emphasizes that these look like maps in the sense of cartography.  The components of $X\setminus \eta(G)$ are the countries and the edges of $G$ are the borders between them.} if $X\setminus \eta(G)$ is homeomorphic to a collection of open discs.
\label{def:topmap}
\end{defn}

Given a surface $X$, a topological graph $G$ may have numerous embeddings, some of which are minimal separating sets, some of which are
non-separating, and some of which are separating but not minimal separating.  Consider the examples in Figure \ref{fig:minsep-ex}.

\begin{figure}
\begin{center}
\begin{tikzpicture}
\draw (-1.6,1.2) -- (1.6,1.2) -- (1.6, 3.2) -- (-1.6,3.2) -- (-1.6, 1.2);
\fill[color=red] (0,2.2) circle [radius=2pt];
\draw[color=red] (-.3, 1.2) -- (0,2.2) -- (-.3,3.2);
\draw[color=red] (.3, 1.2) -- (0,2.2) -- (.3, 3.2);
\draw (0,0) ellipse (1.6 and .9);
\begin{scope}
\path[rounded corners=24pt] (-0.9,0) -- (0, .6) -- (.9,0) (-.9,0) -- (0, -.56) -- (.9,0);
\draw[rounded corners=28pt] (-1.1, .1) -- (0, -.6) -- (1.1,.1);
\draw[rounded corners=28pt] (-.9,0) -- (0,.6) -- (.9,0);
\fill[color=red] (0, -.6) circle [radius=2pt];
\end{scope}
\draw[color=red] (-.13, -.91) arc (-45:45:.45);
\draw[color=red] (.13, -.91) arc (-135:-225:.45);
\draw[dashed,color=red] (-.13, -.91) arc (-135:-225:.45);
\draw[dashed,color=red] (.13, -.91) arc (-45:45:.45);
\end{tikzpicture}
\begin{tikzpicture}
\draw (-1.6,1.2) -- (1.6,1.2) -- (1.6, 3.2) -- (-1.6,3.2) -- (-1.6, 1.2);
\fill[color=red] (0,2.2) circle [radius=2pt];
\draw[color=red] (-1.6,2.2) -- (1.6,2.2);
\draw[color=red] (0, 1.2) -- (0,3.2);
\draw (0,0) ellipse (1.6 and .9);
\begin{scope}
\path[rounded corners=24pt] (-0.9,0) -- (0, .6) -- (.9,0) (-.9,0) -- (0, -.56) -- (.9,0);
\draw[rounded corners=28pt] (-1.1, .1) -- (0, -.6) -- (1.1,.1);
\draw[rounded corners=28pt] (-.9,0) -- (0,.6) -- (.9,0);
\fill[color=red] (0, -.6) circle [radius=2pt];
\end{scope}
\draw[color=red] (-.13, -.91) arc (-45:45:.45);
\draw[dashed,color=red] (-.13, -.91) arc (-135:-225:.45);
\draw[color=red] (0,0) ellipse (1.3 and .6);
\end{tikzpicture}
\begin{tikzpicture}
\draw (-1.6,1.2) -- (1.6,1.2) -- (1.6, 3.2) -- (-1.6,3.2) -- (-1.6, 1.2);
\fill[color=red] (0,2.2) circle [radius=2pt];
\draw[color=red] (-.32,2.2) circle [radius=9pt];
\draw[color=red] (.32,2.2) circle [radius=9pt];
\draw (0,0) ellipse (1.6 and .9);
\begin{scope}
\path[rounded corners=24pt] (-0.9,0) -- (0, .6) -- (.9,0) (-.9,0) -- (0, -.56) -- (.9,0);
\draw[rounded corners=28pt] (-1.1, .1) -- (0, -.6) -- (1.1,.1);
\draw[rounded corners=28pt] (-.9,0) -- (0,.6) -- (.9,0);
\fill[color=red] (0, -.6) circle [radius=2pt];
\end{scope}
\draw[color=red] (-.2,-.6) circle [radius=6pt];
\draw[color=red] (.2, -.6) circle [radius=6pt];
\end{tikzpicture}
\caption{Three embeddings of the figure 8 in the torus.  Each is shown drawn on the standard square torus with sides identified (upper) and on the torus as a donut.  The left is minimal separating and the other two are not.}\label{fig:minsep-ex}
\end{center}
\end{figure}
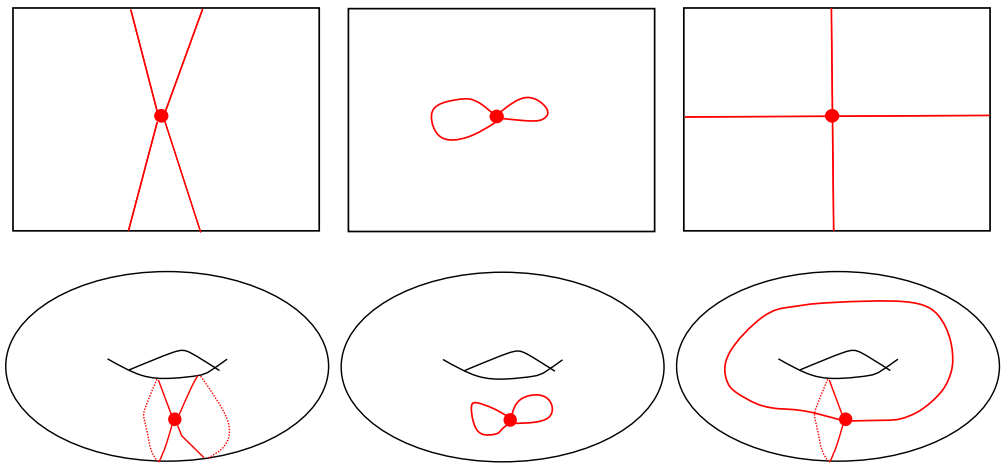

Because of this, we aim first to classify which graph embeddings are minimal separating in a given surface. Then, a topological graph will be realizable as a minimal separating set in a surface $X$ 
if it has an embedding which is minimal separating in $X$. Obtaining an actual list of such graphs can then be done as follows:  Use the classification to obtain a complete list of minimal separating embeddings (up to an appropriate equivalence relation), then construct a list of minimal separating sets up to graph homeomorphism by testing the graphs underlying the embeddings for graph 
isomorphism on a computer.  First we must choose the appropriate equivalence relation.  The following definition of equivalent embeddings is a good place to start.

\begin{defn}\label{def:equiv_emb}  Let $G$ be a topological graph, $X$ be a surface, and $\eta: G\to X$, $\psi:G\to X$ be embeddings of $G$ into $X$.  We say $\eta$ and $\psi$ are equivalent
 embeddings if there is a homeomorphism $f: X\to X$ such that $\psi = f \circ \eta$.  We say that two minimal separating sets $M_1, M_2 \subset X$ are equivalent minimal separating  sets
 if they can be realized as equivalent embeddings of the same topological graph.\end{defn}
From now on, we'll freely use the phrase \textit{minimal separating set} in place of the more precise but tediously long \textit{class of minimal separating sets which are equivalent embeddings}.  
This equivalence relation is close to what we desire, but contains ``redundant'' elements and are a bit more complex to enter in a computer than we would like.  The redundant elements arise from the 
previously mentioned observation from \cite{bernhard}, that given a minimal separating set $M$ in a surface $X_g$, we can construct a related minimal separating set in any surface of genus 
greater than $g$ by adding handles within the connected components of $X_g \setminus M$.  Within such a family of minimal separating sets, we really just want to find the one in the surface of lowest 
genus, and doing so will allow us to omit information about the component genera from our computer program.

To make this precise we introduce the following object associated to a graph embedding:
\begin{defn}\label{def:ribbon}
Let $\eta:G \to X$ be an embedding of a graph $G$ into a surface $X$.
\begin{itemize}
\item The \textit{ribbon graph} $R(\eta)$ is a closed neighborhood of $\eta(G)$ which is small enough to retract to $\eta(G)$.
\item We define $X_{\eta}$ to be the (possibly disconnected) surface obtained by gluing discs to each boundary component of $R(\eta)$ (note that if $\eta$ is a cellular embedding then $X_{\eta}$ is 
homeomorphic to $X$).  We call these discs the \textit{faces} of $R(\eta)$.
\item The genus $g_R$ of $R(\eta)$ is defined to be the genus of $X_{\eta}$ and is given by the formula
\[g_R = \frac{2-V+E-F}{2}\]
where $V$ and $E$ are the numbers of vertices and edges of $G$, and $F$ is the number of faces of $R(\eta)$.
\end{itemize}
\end{defn}
Considering $R(\eta)$, rather than $\eta(G)$, has the following additional benefit for us:  We're eventually going to want to make use of the classification of compact surfaces with boundary.  Given an 
embedding $\eta:G\to X$, $X\setminus \eta(G)$ is not compact.  Considering the union of a component of $X\setminus \eta(G)$ with $\eta(G)$ we do obtain a compact space, but is typically not a surface 
(because of the vertices of degree greater than 2 that $G$ may contain).  However, if we consider $X\setminus R(\eta)^{\circ}$ (where $S^{\circ}$ denotes the interior of a set $S$) each component of this 
space is a compact surface with boundary and there are exactly as many components as there are in $X\setminus\eta(G)$.

Now we are ready to define the final family of sets which we will be enumerating along the way to finding $\mathcal{M}_g,\mathcal{L}_g,$ and $\mathcal{C}_g$.
\begin{defn}\label{def:rg}
$\mathcal{R}_g$ is the set of \textit{connected} ribbon graphs (up to homeomorphism) which can be realized as the ribbon graph of a minimal separating set in $X_g$, but not in $X_{g-1}$.
\end{defn}

Now we relate $\mathcal{R}_g$ to $\mathcal{C}_g$:
\begin{lemma}\label{lem:rgcg}
A topological graph $G$ is an element of $\mathcal{C}_g$ if and only if there is an embedding $\eta:G\to X_g$ with $R(\eta)\in\mathcal{R}_g$, but for any $\hat{g}<g$ there is no embedding $\psi:G\to 
X_{\hat{g}}$ with $R(\psi)\in \mathcal{R}_{\hat{g}}$.
\end{lemma}

\begin{proof}
Let $G\in \mathcal{C}_g$.  By definition of $\mathcal{C}_g$, there is an embedding $\eta:G\to X_g$ where $\eta(G)$ is a minimal separating set, and $R(\eta)$ is the ribbon graph of a minimal separating 
set in $X_g$.  The topological graph underlying a ribbon graph is unique up to homeomorphism, so $R(\eta)$ cannot be realized as a minimal separating set in $X_{\hat{g}}$ for any $\hat{g}<g$ (that would
 contradict the minimal genus condition in the definition of $\mathcal{C}_g$).  Thus $R(\eta)\in \mathcal{R}_g$.  Similarly, if there was an embedding $\psi:G\to X_{\hat{g}}$ with $\hat{g}<g$ and $R(\psi) 
 \in \mathcal{R}_{\hat{g}}$, we could recover from $R(\psi)$ a minimal separating embedding of $G$ into $X_{\hat{g}}$, contradicting the minimal genus condition on $\mathcal{C}_g$.

Now suppose that we have a surface $X$ and embedding $\eta:G\to X$ with $R(\eta)\in\mathcal{R}_g$, but there is no surface $Y$ and embedding $\psi:G\to Y$ with $R(\psi)\in\mathcal{R}_{\hat{g}}$ for any 
$\hat{g}<g$.  Since $R(\eta)\in\mathcal{R}_g$, we know that $G$ has a minimal separating embedding in $X_g$, and thus $G\in \mathcal{M}_g$.  Since $G$ is connected, $G\in\mathcal{C}_{\bar{g}}$ for 
some $\bar{g}\le g$.  If $\bar{g}<g$, that would mean it would have an embedding $\psi:G\to Y$ with $\psi(G)$ minimal separating in $X_{\bar{g}}$, which would mean $R(\psi) \in \mathcal{R}_{\hat{g}}$ for 
some $\hat{g}\le \bar{g}<g$, giving a contradiction.
\end{proof}

Since the underlyling surface of an embedding is oriented, given an edge $e$ of $G$, it makes sense to speak of the two sides of $\eta(e)$.  Thinking of starting at a point in $\eta(e)$ and
walking along $\eta(e)$ towards one of the two endpoints of $\eta(e)$, the two sides are the left and right of the walker.

We can now immediately state the following result about when a graph embedding in a surface is minimal separating:

\begin{lemma} \label{lem:twocoloring} An embedding $\eta(G)$ of a graph $G$ in a surface $X$ is \textit{minimal separating} in $X$ if and only if $X\setminus \eta(G)$ has two connected
components, $A$ and $B$, and for every edge $e$ of $G$, $\eta(e)$ has $A$ on one side and $B$ on the other.\end{lemma}

\begin{proof}
If $\eta(G)$ is minimal separating, then certainly $X\setminus \eta(G)$ has at least two components, which we call $A,B,\ldots$.  By minimality, for any edge $e$ of $G$, $(X\setminus \eta(G))\cup \eta(e)$
is not separated.  Thus $\eta(e)$ is incident to all the components of $X\setminus \eta(G)$. Thus there are exactly two components $A$ and $B$ of $X\setminus \eta(G)$ with $A$
lying to one side of $\eta(e)$ and $B$ lying to the other.

Now assume $\eta(G)$ is an embedding with $X\setminus \eta(G)$ having two disjoint components $A$ and $B$, and for all edges $e$ of $G$, $\eta(e)$ is incident to $A$ on one side and $B$
on the other.  $\eta(G)$ is clearly separating so we just need to address minimality.  Since each edge $e$ satisfies $\eta(e)$ is incident to $A$ on one side and $B$ on the other, for any point
$x \in \eta(e)$, there is a path in $A$ from any point in $A$ to $x$, and a path in $B$ from any point in $B$ to $x$.  Thus $X\setminus (\eta(G)\setminus x)$ is connected and $\eta(G)$ is
minimal separating.
\end{proof}

\section{Non-Cellular Embeddings and Disconnected Graphs}
We wish to use powerful existing tools for working with cellular embeddings of connected graphs, but in general, the graph embeddings which give us minimal separating sets are neither
connected nor cellular.  As an example, consider the left embedding from Figure \ref{fig:minsep-ex}.

If a graph embedding $\eta:G\to X_g$ is cellular, it's simple to tell from $R(\eta)$ whether $\eta(G)$ is minimal separating.  Since $X_g=X_{\eta}$ for cellular $\eta$, a cellular embedding is 
minimal separating precisely when $R(\eta)$ has two faces, and each each edge of $R(\eta)$ is incident to both faces.  In this section we will extend this simple rule to a determination of precisely 
which ribbon graphs are elements of $\mathcal{R}_g$, and we will determine the elements of the $\mathcal{L}_{g}$s in terms of the elements of the $\mathcal{C}_{g}$s.  These results will allow us to 
determine the elements of the $\mathcal{M}_{g}$s, $\mathcal{L}_{g}$s, $\mathcal{C}_{g}$s, and $\mathcal{R}_{g}$s, since Lemma \ref{lem:rgcg} gives a 
method to obtain the elements of the $\mathcal{C}_{g}$s from the $\mathcal{R}_{g}$s and Definition \ref{def:graphsets} allows us to determine the elements of the $\mathcal{M}_{g}$s from the 
$\mathcal{L}_{g}$s.

\begin{defn}\label{def:facecoloring}
Let $R$ be a ribbon graph and $n$ be a positive integer.  A \textit{proper $n$-face-coloring} of $R$ is a function that maps each face of $R$ to an element of $\{1,2,\ldots, n\}$ such that no edge of $R$ 
has both sides incident to faces of the same color.
\end{defn}

We can restate Lemma \ref{lem:twocoloring} in the language of face-colorings as:  An embedding $\eta(G)$ is minimal separating in a surface $X$ if and only if $X\setminus \eta(g)$ has two connected 
components and $R(\eta)$ can be propery 2-face-colored.

\begin{lemma}\label{lem:lgcomponents}
Let $G\in \mathcal{L}_g$, $\eta:G\to X_g$ be a minimal separating embedding of $G$ into $X_g$, and $A,B$ be the connected components of $X\setminus R(\eta)^{\circ}$.  Then $A$ and $B$ have genus 0.
\end{lemma}

\begin{proof}
Without loss of generality, assume $A$ has genus greater than 0, then it contains a handle.  We can `cut' this handle with a non-separating curve and call the resulting surface $A'$.  
Since the boundary of $A'$ is identical to that of $A$, and $A'$ is still connected, we can glue $A'$ onto $R(\eta)$ exactly as $A$ was attached to $R(\eta)$, and by gluing $B$ back onto $R(\eta)$ 
as well, we see that $R(\eta)$ retracts to a minimal separating embedding of $G$ in this new surface, but since a handle was cut, this surface has genus $g-1$, contradicting the 
assumption that $G\in \mathcal{L}_g$.
\end{proof}

The proof of Lemma \ref{lem:lgcomponents} gives the following corollary
\begin{corollary}\label{cor:rgcomponents}
Let $R$ be a connected ribbon graph.  $R\in\mathcal{R}_g$ if and only if $R$ can be realized as the ribbon graph of a minimal separating set in $X_g$ and $A$ and $B$ have genus 0 (where $A$ and $B$ are the 
connected components of $X\setminus R^{\circ}$).
\end{corollary}

\begin{corollary}\label{cor:facecount}
Let $R$ be a connected ribbon graph and $\eta$ an embedding with $R=R(\eta)$ and assume that the faces of $X_{\eta}$ can be properly 2-face-colored.  Then $R\in\mathcal{R}_g$ if and only if 
\[F = 2+g-g_R\]
where $F$ is the number of faces of $R$.  Since $F\ge 2$ this gives $g_R\le g$.
\end{corollary}

\begin{proof}
Suppose $R\in \mathcal{R}_g$ and let $\eta$ be a minimal separating embedding of a graph $G$ into $X_g$ such that $R=R(\eta)$.  Let $A,B$ be the connected components of $X_g\setminus R^{\circ}$. By 
Corollary \ref{cor:rgcomponents}, $A$ and $B$ both have genus 0 and are homeomorphic to spheres with some number of discs removed (one disc removed per component of their boundary with $R$).  Knowing this 
we will reconstruct $X_g$ from $X_{\eta}$.

Partition the faces of $X_{\eta}$ into two sets $F_A$ and $F_B$ according to whether they were glued to $R(\eta)$ along its boundary with $A$ or $B$.  Now select a disc $D\in F_A$ as a `base disk' and we 
will join each other disc $D' \in F_A$ to $D$ be attaching a handle with one end in $D$ and the other in $D'$.  By connecting each face in $F_A$ to $D$ in this way, we have now connected these discs into a 
single connected surface.  Cutting any of these handles will disconnect the surface (separating the the one disc from the base disc $D$ and all the others), so we haven't added any genus.  By the 
classifcation of compact surfaces with boundary (\cite{kinsey}) the resulting surface is homeomorphic to a sphere with one boundary component for each boundary with $R(\eta)$.  That is, it is homeomorphic 
to $A$.  Repeating this procedure for $F_B$ gives a surface homeomorphic to $B$.  This means that the handle gluing process described here converts $X_{\eta}$ into a surface homeomorphic to $X_{g}$.

Consider the effect of this operation on the genus of $X_{\eta}$. We added $|F_A|-1$ handles to reconstruct $A$, and $|F_B|-1$ handles to reconstruct $B$.  Since $R$ is connected by assumption, each 
handle addition increased the genus of $X_{\eta}$ by 1, so we have
\[g = g_R + (|F_A|-1)+(|F_B|-1) = g_R + F-2\]
giving $F=2+g-g_R$ as desired.

For the other direction, assume that $R=R(\eta)$ for some $\eta$, the faces of $X_{\eta}$ can be properly 2-face-colored, and $X_{\eta}$ has $2+g-g_R$ faces.  Do exactly the procedure described above to 
connect all the faces in each color class into a single connected component by adding handles.  As shown above this results in a $X_g$, with $R$ embedded so that it contracts to a minimal separating set  
by Lemma \ref{lem:twocoloring}.  This gives $R\in\mathcal{R}_{g'}$ for some $g'$, and applying the first part of the proof gives $g=g'$.
\end{proof}

This gives a precise condition for whether a connected ribbon graph $R$ is in $\mathcal{R}_g$ for some $g$ (that it can be properly 2-face-colored), and provides a formula for the value of $g$ in terms of 
$g_R$ and the number of faces of $R$.  Next we 

We obtain the following results relating $\mathcal{L}_g$ with the sets $\mathcal{C}_{h}$ for
$h\le g$.

\begin{theorem}\label{thm:simple_union}Let $G$ be a topological graph which is the disjoint union of
graphs $G_1$ and $G_2$. Then
\[
G \in \mathcal{L}_g \quad \Longleftrightarrow \quad
G_i \in \mathcal{L}_{g_i} \textrm{ with }g= g_1+g_2 +1 \,.
\]
\end{theorem}

\begin{proof}
Let $G \in \mathcal{L}_g$ and $\eta$ be a minimal separating embedding of $G$ into a surface $X$ with
genus $g$.  We'll let $\eta_1$ and $\eta_2$ be the restriction of $\eta$ to $G_1$ and $G_2$
respectively. Consider the decomposition of $X$ into $R(\eta)=R(\eta_1)\cup R(\eta_2)$ and
the two connected components of
$X\setminus R(\eta)^{\circ}$ as shown in Figure \ref{fig:minsep-decomp}.  As in the figure, we color one
component white, labelling it $X_w$, and the other black, labelling it $X_b$.  By Lemma \ref{lem:lgcomponents}, both components are homeomorphic to spheres with one missing disc per boundary component with $R(\eta)$.  
By cutting both the black component along the red curve $C_b$, and
cutting the white one along the red curve $C_w$, and ``plugging'' the new holes by gluing on
discs, when we glue the spheres back onto
$R(\eta)$, we see that we now have two surfaces, $X^{\prime}$ and $X^{\prime\prime}$ with $R(\eta_1)$ the ribbon graph of an embedding in $X^{\prime}$ and $R(\eta_2)$ the ribbon graph of an embedding in 
$X^{\prime\prime}$.

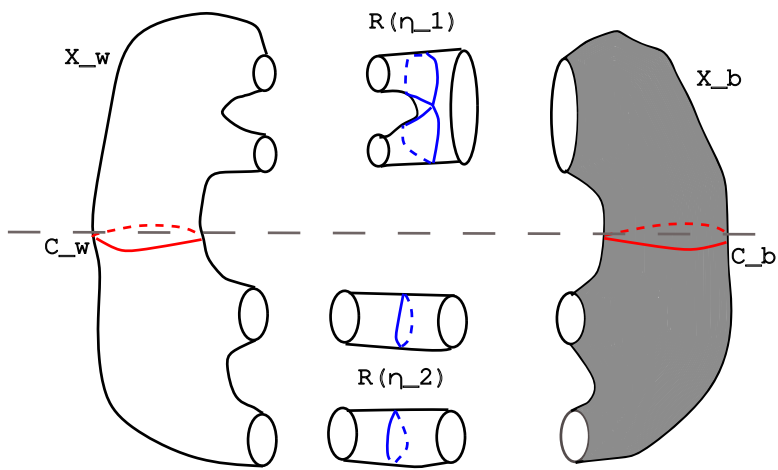
\begin{figure}
\begin{center}
\begin{tikzpicture}
\filldraw[fill=gray, draw= black] (3,1.5) -- (3, -1.2) arc (90:-90:0.3 and .4) arc (90:-90:0.5) arc (90:-90:0.3 and .4) arc (-90:0:2 and 3.8) arc (0:90:2 and 3.5) arc (90:-90:0.3 and 1);
\draw[color=red] (-3,0) arc (360:180:1 and .4);
\draw[dashed, color=red] (-3,0) arc (0:180:1 and .4);
\draw[color=red] (3,0) arc (180:360:1 and .4);
\draw[dashed, color=red] (3,0) arc (180:0:1 and .4);
\draw (-3,1.6) ellipse (0.3 and .4);
\draw (-3, 3.4) ellipse (0.3 and .4);
\draw (-3, -1.6) ellipse (0.3 and .4);
\draw (-3, -3.4) ellipse (0.3 and .4);
\draw (3, 2.5) ellipse (0.3 and 1);
\draw (3, -1.6) ellipse (0.3 and .4);
\draw (3, -3.4) ellipse (0.3 and .4);
\draw (1, 2.5) ellipse (0.3 and 1);
\draw (3, 1.5) -- (3, -1.2);
\draw (3, 3.5) arc (90:0:2 and 3.5);
\draw (3, -3.8) arc (-90:0:2 and 3.8);
\draw (3,-3) arc (-90:90:0.5);
\draw (-1, 3.8) arc (90:270:0.3 and .4);
\draw[dashed] (-1, 3.8) arc (90:-90:0.3 and .4);
\draw (-1, 1.2) arc (270:90:0.3 and .4);
\draw[dashed] (-1,1.2) arc(-90:90:0.3 and .4);
\draw (-1, -1.2) arc (90:270:0.3 and .4);
\draw[dashed] (-1, -1.2) arc (90:-90:0.3 and .4);
\draw (-1, -3) arc (90:270:0.3 and .4);
\draw[dashed] (-1, -3) arc (90:-90:0.3 and .4);
\draw (1, -1.6) ellipse (0.3 and .4);
\draw (1, -3.4) ellipse (0.3 and .4);
\draw (-1,2) arc (-90:90:0.5);
\draw (-3, 3) arc (90:270:0.5);
\draw (-3,-2) arc (90:270:0.5);
\draw (-1, -2) -- (1,-2);
\draw (-1, -1.2) -- (1,-1.2);
\draw (-1, -3) -- (1,-3);
\draw (-1, -3.8) -- (1,-3.8);
\draw (-3, 1.2) -- (-3,-1.2);
\draw (-3, 3.8) arc (90:270:2 and 3.8);
\draw (-1,3.8) -- (1, 3.5);
\draw (-1, 1.2) -- (1, 1.5);
\draw[dashed, color = blue] (0,-3.8) arc (-90:90:0.3 and .4);
\fill[color=blue] (-0.3, -3.4) circle[radius=2pt];
\draw[color=blue] (0,-3) arc(90:270:0.3 and .4);
\fill[color=blue] (-0.3, -1.6) circle[radius=2pt];
\draw[color=blue] (0,-1.2) arc (90:270:0.3 and .4);
\draw[dashed, color=blue] (0,-1.2) arc(90:-90:0.3 and .4);
\fill[color=blue] (0,2.5) circle[radius=2pt];
\draw[color=blue] (0,3.65) arc (160:200:1.7);
\draw[color=blue] (0,2.5) arc (160:200:1.7);
\draw[color=blue] (0,2.5) -- (-.6, 2.81);
\draw[color=blue] (0,2.5) -- (-.6, 2.19);
\node[inner sep=0] (a) at (0,3.65) {};
\node[inner sep=0] (b) at (-.6, 2.8) {};
\node[inner sep=0] (c) at (-.6, 2.2) {};
\node[inner sep=0] (d) at (0, 1.35) {};
\draw[dashed] (-6,0) -- (6,0);
\path[dashed, color=blue]
(a) edge[bend left] (b)
(c) edge[bend left] (d);
\node at (0, 3.65) [label={above:\footnotesize{$R(\eta_1)$}}] {};
\node at (0, -2.0) [label={below:\footnotesize{$R(\eta_2)$}}]{};
\node at (-4.9,-.2) [label={left:\footnotesize{$C_w$}}]{};
\node at (4.9, -.2) [label={right:\footnotesize{$C_b$}}]{};
\node at (-4.5, 2.8) [label={left:\footnotesize{$X_w$}}]{};
\node at (4.5, 2.8) [label={right:\footnotesize{$X_b$}}]{};
\end{tikzpicture}\qquad
\begin{tikzpicture}
\node at (0, 3.5) [label={above:\footnotesize{$G_1$}}] {};
\node at (0, -.5) [label={above:\footnotesize{$G_2$}}] {};
\fill[color=blue] (0,2.5) circle[radius=2pt];
\draw[color=blue] (0,3) circle[radius=.5];
\draw[color=blue] (0,2) circle[radius = .5];
\draw[color=blue] (0,-1) circle[radius=.5];
\draw[color=blue] (0,-2.2) circle[radius=.5];
\fill[color=blue] (-.5,-1) circle[radius=2pt];
\fill[color=blue] (-.5, -2.2) circle[radius=2pt];
\end{tikzpicture}
\caption{Left: An example where $G_1$ is a single vertex with two attached loops, $G_2$ is a pair of vertices, each with a loop attached, and gluing the components together recovers a minimal separating embedding of
 $G= G_1\cup G_2$ in a surface $X$. $X_w$ and $X_b$ are the connected components of $X\setminus R(G)$, which have been colored white and black respectively. Right: $G_1$ and $G_2$.}\label{fig:minsep-decomp}\end{center}\end{figure}

Note that each edge of $R(\eta_1)$ in $X^{\prime}$ is still incident to the black and white components, and similarly for the edges of $R(\eta_2)$ in $X^{\prime\prime}$.  Thus the induced embeddings of 
$G_1$ and $G_2$ are minimal separating in $X^{\prime}$ and $X^{\prime\prime}$ respectively.  Then there must exist $g_1,g_2$ such that $G_1 \in \mathcal{L}_{g_1}$ and $G_2\in\mathcal{L}_{g_2}$. Futhermore,
 $X^{\prime}$ has genus at least $g_1$ and $X^{\prime\prime}$ has genus at least $g_2$.  We can recover $X$ from $X^{\prime}$ and $X^{\prime\prime}$ by taking a connected sum (glued together within the 
black component of each surface, and adding a handle to connect the white components).  Since genus is additive under connected sums, and adding a handle increase genus by one, we have:
\[g \ge 1+g_1+g_2.\]

Now we let $G_1,G_2$ be topological graphs with $G_1\in\mathcal{L}_{g_1}$ and $G_2\in\mathcal{L}_{g_2}$ and let $G$ be their disjoint union.  Let $\eta_1:G_1\to X_{g_1}$ and $\eta_2:G_2\to X_{g_2}$ be 
minimal separating embeddings.  Now we construct a minimal separating embedding of $G$ in $X_{g_1+g_2+1}$.  For both $\eta_1$ and $\eta_2$, color one of the connected components of 
$X_{g_i}\setminus R(\eta_i)^{\circ}$ black and the other white.  Now take the connected sum of $X_{g_1}$ and $X_{g_2}$ by gluing together within the black component of each surface.  Add a handle connecting 
the two white components.  We've now constructed a surface of genus $g_1+g_2+1$, and embedded $G$ into it so that each edge of $G$ is incident to the black region on one side and the white region on the 
other.  Thus we have a minimal separating embedding of $G$ into $g_1+g_2+1$, so $G \in \mathcal{L}_g$ for some $g$ satisfying
\[g\le 1+g_1+g_2.\]
Combining the inequalities gives the desired result.
\end{proof}

\begin{corollary}\label{cor:connected-components}
Let $G \in \mathcal{L}_g$ and let $G_1,\ldots, G_k$ be the connected components of $G$.  Then there exist $g_1,\ldots, g_k$ such that for each $i \in 1,2,\ldots, k$ we have $G_i \in \mathcal{C}_{g_i}$ and
 $g = (k-1)+\sum_{i=1}^{k}g_i $.
\end{corollary}

\begin{proof}
The proof follows from Theorem \ref{thm:simple_union} by induction on the number of connected components of $G$.
\end{proof}
With this result, we can give a formula for the $|\mathcal{L}_g|$ in terms of the $|\mathcal{C}_{\hat{g}}|$ for $\hat{g}\le g$.

\begin{corollary}\label{cor:lgcount}
Let $K_g = \{(k_0,\ldots, k_g) : g= (\sum_{i=0}^{g}(i+1)k_i) -1\}.$  Then
\[|\mathcal{L}_g| = \sum_{(k_0,\ldots, k_g) \in K_g} \prod_{i=0}^{g} {\lvert C_i\rvert + k_i -1 \choose k_i}\]
\end{corollary}

\begin{proof}
From Corollary \ref{cor:connected-components} we know that the elements of $\mathcal{L}_g$ are precisely the graphs whose connected components $G_1,\ldots, G_k$ satisfy $G_i \in \mathcal{C}_{g_i}$ for 
some $g_i$ and 
\[g_1+g_2+\ldots +g_{\ell}+ (\ell-1)=g.\]
For each $i\in \{0,1,\ldots, g\}$ let $k_i = |\{G_1,\ldots, G_{\ell}\}\cap \mathcal{C}_{i}|$, the number of components of $G$ which are in $\mathcal{C}_i$.  Then we have:
\[g = \left(\sum_{i=0}^{g}i k_i\right) + (\ell -1)\]
and since each component is in exactly one of the $\mathcal{C}_i$ we have $\ell = \sum_{i=0}^{g}k_i$.  This gives
\begin{align*}
g &= \left(\sum_{i=0}^{g} i k_i\right) + \left(\sum_{i=0}^{g}k_i\right) -1\\
  &= \left(\sum_{i=0}^{g} (i+1)k_i \right) -1
\end{align*}
So, $G\in \mathcal{L}_g$ if and only if this equality holds.  Now we need to determine how many graphs have exactly $k_i$ connected components in each $\mathcal{C}_i$.

For $G\in\mathcal{L}_g$ define $K(G):=(k_0,\ldots, k_g)$ where $k_i$ is the number of connected components of $G$ in $\mathcal{C}_i$. Since two graphs $G$ and $G'$ are isomorphic if and only if there is a 
bijection between their connected components which takes distinct components of $G$ to distinct isomorphic components of $G'$, the number of graphs $G$ with $K(G) = (k_0,\ldots, k_g)$ is equal to the 
number of was to choose (with replacement) $k_0$ elements from $\mathcal{C}_0$ times the number of ways to choose $k_1$ elements from $\mathcal{C}_1$ \dots times the number of ways to choose $k_g$ 
elements from $\mathcal{C}_g$.  Since the number of ways to choose $k$ objects with replacement from a set of size $n$ is equal to ${n+k-1\choose k}$ \cite{stanley}, we have the desired formula
\[\lvert \mathcal{L}_g\rvert = \sum_{(k_0,\ldots, k_g)\in K_g}\prod_{i=0}^{g}{\lvert C_i\rvert +k_i -1 \choose k_i}.\]
\end{proof}

With these results relating the $\mathcal{R}_g$ to the $\mathcal{C}_g$, and the $\mathcal{C}_g$ to the $\mathcal{L}_g$, we can now turn our attention to a method of finding all the elements of the 
$\mathcal{R}_g$.

\section{Topological and Combinatorial Maps}
Given a ribbon graph or a cellular embedding, there is discrete data called a combinatorial map which allows us to recover the ribbon graph up to homeomorphism or the embedding up to embedding equivalence.
These objects are equivalently known as rotation systems, fat graphs, ribbon graphs\footnote{However,
we use the term \textit{ribbon graph} exclusively as in Definition \ref{def:ribbon}.}, or dessins d'enfants (\cite{hidalgo}, \cite{liu},\cite{lando}).
Precise definitions vary a bit from source to source based on the types of embeddings of interest to the author (particularly if the author is interested in non-orientable surfaces).  We use the following definition.

\begin{defn}\label{def:combinatorialmap} \cite{lando} A \emph{combinatorial map} is an ordered
triple\footnote{Or, equivalently, an ordered pair $(\sigma,\alpha)$.}
 $(\sigma,\alpha,\varphi)$\footnote{The letters $\sigma,\alpha,\varphi$ come from the French \textit{sommet} for vertex, \textit{arc} for edge and \textit{face} for face\cite{jacques}.} of permutations in $S_{2n}$ (the symmetric group on $2n$ elements), such that \\[0.15cm]
 i)\;\; $\alpha$ is an involution with no fixed points.\\[0.15cm]
ii)\;\; The permutation group $\langle \sigma, \alpha\rangle=\langle \sigma, \alpha,\varphi\rangle$ acts transitively on the set $\{1,2,\ldots, 2n\}$.\\[0.15cm]
iii)\;\; $\varphi\circ\alpha\circ\sigma=1$.
\end{defn}

\vskip 0.0in\noindent
The second condition is equivalent to the underlying graph being connected.

We give an example of the correspondence between combinatorial maps and connected cellular embeddings 
(see Definition \ref{def:topmap}). For a proof of this, see \cite{lando}. Given a connected graph $G$
and a cellular embedding of $G$ into an orientable surface $X$, the following four-step reasoning
associates a combinatorial map  to our embedding of $G$ (see Figure \ref{fig:map-ex}).
\begin{enumerate}
\item Assign labels to the edge ends and mark them on the surface $X$ as described above.
\item Define $\alpha$ as the involution that swaps the two ends of each edge. The number $c(\alpha)$ of cycles of $\alpha$ gives the number of edges.
\item Define $\sigma$ as the permutation that assigns the next (counterclockwise) edge at each vertex. The number $c(\sigma)$ of cycles of $\sigma$ gives the number of vertices.
\item Then $\varphi=(\sigma \alpha)^{-1}$ turns out to be (see below) the `boundary walk'. That is: the cycles of $\varphi$ are precisely the edges that form the boundary
of each component (or cell)  of the complement of the graph. The number $c(\varphi)$ of cycles of $\varphi$ gives the number of faces.
\end{enumerate}

\begin{figure}
\begin{center}
\begin{tikzpicture}
\foreach \a in {1,2,3}
\fill ({\a*120+30}:1.5) circle[radius=2pt];
\draw [rounded corners, thick] (150:1.5) -- (210:1.5) -- (270:1.5);
\draw [rounded corners, thick] (150:1.5) -- (210:.2) -- (270:1.5);
\draw [thick] (150:1.5) -- (30:1.5);
\draw [thick] (270:1.5) -- (30:1.5);
\draw [thick] (30:1.5).. controls (52:3.0) and (8:3.0).. (30:1.5);
\end{tikzpicture}\qquad
\begin{tikzpicture}
\foreach \a in {1,2,3}
\fill ({\a*120+30}:1.5) circle[radius=2pt];
\draw [rounded corners, thick] (150:1.5) -- (210:1.5) -- (270:1.5);
\draw [rounded corners, thick] (150:1.5) -- (210:.2) -- (270:1.5);
\draw [thick] (150:1.5) -- (30:1.5);
\draw [thick] (270:1.5) -- (30:1.5);
\draw [thick] (30:1.5).. controls (52:3.0) and (8:3.0).. (30:1.5);
\node at (125:1.2) [label={below:\footnotesize{$1$}}]{};
\node at (150:1.6) [label={below right:\footnotesize{$5$}}]{};
\node at (142:1.4) [label={below left:\footnotesize{$3$}}]{};
\node at (270:1.1) [label={left: \footnotesize{$4$}}]{};
\node at (35:1.7) [label={left:\footnotesize{$2$}}]{};
\node at (272:1.3) [label={above:\footnotesize{$6$}}]{};
\node at (265:1.3) [label={right :\footnotesize{$8$}}]{};
\node at (42:1.3) [label={below :\footnotesize{$7$}}]{};
\node at (21:1.7) [label={above :\footnotesize{$9$}}]{};
\node at (5:1.6) [label={above:\footnotesize{$10$}}]{};
\end{tikzpicture}
\caption{An embedding of a graph in the plane is shown on the left, and the same embedding equipped with a labelling of the edge ends is shown on the right.  Each label is immediately
clockwise of its corresponding edge end.  The corresponding combinatorial map is $\sigma=(1,3,5)(4,8,6)(2,7,10,9),\alpha = (1,2)(3,4)(5,6)(7,8)(9,10), \varphi = (1,6,7)(2,10,8,3)(4,5)(9)$.}
\label{fig:map-ex}
\end{center}\end{figure}
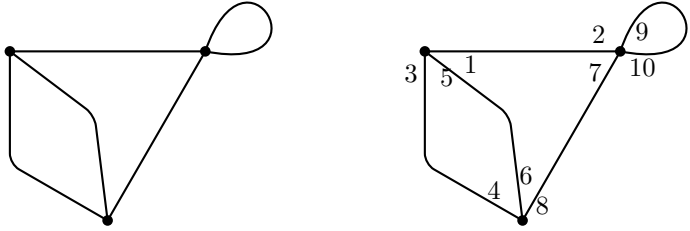

The procedure above certainly gives a way to associate a triple of permutations in $S_{2n}$ to any
connected graph with $n$ edges.  Since $G$ is connected it's possible to walk from any edge end to
any other one by travelling to the opposite end of the edge (applying $\alpha$), then moving to another
edge end incident to the current vertex (applying $\sigma$) and repeating.  This shows that
$\langle \sigma,\alpha\rangle$ acts transitively on $\{1,\ldots, 2n\}$.  It remains to verify that
$\varphi$ in Definition \ref{def:combinatorialmap} is really a boundary walk.  We illustrate
this with a simple picture looking at two vertices with an edge between them, see Figure
\ref{fig:vert-label}.

Given a combinatorial map and associated ribbon graph, we know that the cycles of the permutations defining the map correspond to vertices, edges, and faces of the ribbon graph.  This leads to the following 
convenient fact:
\begin{lemma}\label{lem:map_genus} Given a combinatorial map $(\sigma,\alpha,\varphi)$ of a cellular embedding $\eta$, with ribbon graph $R_{\eta}$, the genus $g_R$ of $X_{\eta}$ satisfies
\[2-2g_R = V-E+F=c(\sigma)-c(\alpha)+c(\varphi).\]
\end{lemma}

\begin{proof} This follows directly from the four points listed above.
\end{proof}

Notice that given any connected graph $G$ and an embedding $\eta:G\to X$ (cellular or not) we can construct a combinatorial map with the same ribbon graph as $\eta$ by steps 1-3 of the procedure described
 above, but if the embedding isn't cellular then point 4 will not be true.  Also, just as the possible existence of degree 2 vertices allows for combinatorially distinct graphs to be homeomorphic, 
 they allow for combinatorial maps with different cycle types to be associated to the same ribbon graph.  
 So, for any embedding of a connected graph we can talk about \textit{a} combinatorial map associated to the embedding.  We'd 
 like to refer to \textit{the} combinatorial map associated to an embedding, but we must address degree 2 vertices and the fact that there are typically many ways to assign labels to the edge ends.  Degree
  2 vertices are easily handled:  For every ribbon graph which is not the annulus we only consider associating to it combinatorial maps $(\sigma,\alpha,\varphi)$ where $\sigma$ has no 2-cycles (so the 
  underlying graph has no degree 2 vertices).  

To deal with the matter of a choice of labelling, there is the notion of \textit{isomorphism} of combinatorial maps, again following \cite{lando}.  It captures the idea that two
combinatorial maps arising from different choices of labelling on the same topological map are really \textit{the same object}.  In terms of the $S_{2n}$ action on $\{1,\ldots, 2n\}$,
relabeling can be thought of as conjugation by the relabelling permutation $\rho$, and we arrive at the following definition:

\begin{defn}\label{def:map-iso}
We say $\rho \in S_{2n}$ is an isomorphism between combinatorial maps $(\sigma_1,\alpha_1,\varphi_1)$ and $(\sigma_2,\alpha_2,\varphi_2)$ if
\[\rho^{-1}\sigma_1\rho = \sigma_2,\qquad \rho^{-1}\alpha_1\rho = \alpha_2.\]
And this, of course, implies that $\rho^{-1}\varphi_1\rho = \varphi_2$.
\end{defn}

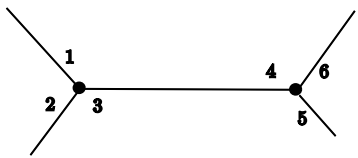
\begin{figure}
\begin{center}
\begin{tikzpicture}
\draw(-1,0) -- (1,0);
\fill (-1,0) circle [radius=2pt];
\fill (1,0) circle [radius=2pt];
\draw (-1.7,1.3) -- (-1,0);
\draw (-1.7,-1.3) -- (-1,0);
\draw (1,0) -- (1.7,1.3);
\draw (1,0) -- (1.7,-1.3);
\node at (179:1.1) [label={above:\footnotesize{$1$}}]{};
\node at (-190:1.3) [label={below:\footnotesize{$2$}}]{};
\node at (-198:0.75) [label={below:\footnotesize{$3$}}]{};
\node at (-16:0.75) [label={above:\footnotesize{$4$}}]{};
\node at (0:1.1) [label={below:\footnotesize{$5$}}]{};
\node at (12:.9) [label={right:\footnotesize{$6$}}]{};
\end{tikzpicture}
\caption{With counterclockwise as the positive orientation we have $\alpha(4)=3$ and $\sigma(3) = 1$, so $(\alpha\circ\sigma)(4) = 1$.  Then $(\alpha\circ\sigma)^{-1}(1)=4$, and the edge with label 4 is the
edge following the edge with label 1 in counterclockwise order about their shared face.}\label{fig:vert-label}
\end{center}\end{figure}

It follows from Theorems 1 and 2 of \cite{edmonds} there is a one-to-one correspondence between homeomorphism classes of connected ribbon graphs and isomorphism classes of combinatorial maps without 
degree 2 vertices (except when the underlying graph is the circle graph). From now on we'll freely refer to elements of $\mathcal{R}_g$ as combinatorial maps $(\sigma,\alpha,\varphi)$ even though each 
element actually corresponds to an isomorphism class of combinatorial map.  Table \ref{tab:gloss1} gives a glossary for translating between some values associated to a ribbon graph and the corresponding 
combinatorial map.

\begin{table}\begin{center}\begin{tabular}{|c|c|}
Ribbon Graph Property & Combinatorial Map Property\\
\hline
Number of Vertices & $c(\sigma)$\\
\hline
Number of Edges & $c(\alpha)$\\
\hline
Number of Faces & $c(\varphi)$\end{tabular}\caption{Glossary for translating between ribbon graphs and combinatorial maps.  In general some care must be taken with this because of the `topological 
invisibility' of degree 2 vertices, but in all the cases of interest to us degree 2 vertices have been forbidden.}\label{tab:gloss1}\end{center}\end{table}

There is one more attribute of combinatorial and topological maps which we must introduce, the dual.  This is generalization of the dual of a
graph embedded in the plane (or sphere) to a graph cellularly embedded in an arbitrary surface.  As we'll see, considering the duals of elements of $\mathcal{R}_g$ will allow us to shrink the search 
space of candidates for elements of $\mathcal{R}_g$.

\begin{defn}\label{def:map-dual}
The geometric dual of a topological map $M$ is the map $M^{*}$ constructed as follows:  We place a single vertex of $M^{*}$ inside each face of $M$.  Then, for every edge $e$ in $M$, we place an edge in $M^{*}$
running between the vertices of $M^{*}$ corresponding to the faces of $M$ incident to $e$, with the new edge passing through the midpoint of $e$.
\end{defn}

\begin{corollary}\cite{lando} If $(\sigma,\alpha,\varphi)$ is a combinatorial map with $M$ the corresponding topological map, then $(\varphi,\alpha,\sigma)$ is the combinatorial map corresponding 
to $M^{*}$.
\end{corollary}

\begin{proof} Definition \ref{def:map-dual} immediately implies that the faces of the dual of $M$ correspond to the vertices of $M$ and so forth.
\end{proof}

\begin{lemma}
Let $R$ be a ribbon graph with underlying graph $G$.  Then $R\in\mathcal{R}_g$ for some $g$ if and only if $R$ has a bipartite\footnote{We say a graph or graph embedding is \textit{bipartite} if the 
vertices can be partitioned into two color classes so that every edge is incident to a vertex in each color class.} dual.
\end{lemma}

\begin{proof}
Our restatement of Lemma \ref{lem:twocoloring} in terms of face coloring tells us that we must be able to properly 2-face-color $R$ if it is an element of any $\mathcal{R}_g$, and then Corollary 
\ref{cor:facecount} tells precisely the value of $g$ such that $R\in\mathcal{R}_g$, in terms of $g_R$ and the number of faces of $R$.  Thus $R$ is in some $\mathcal{R}_g$ if and only if it can 
be properly 2-face-colored.  Since taking the dual of $R$ exchanges vertices with faces we conclude that $R \in \mathcal{R}_g$ for some $g$ if and only if it's dual is bipartite.
\end{proof}

We can now restate Corollary \ref{cor:facecount} in terms of dual combinatorial maps:
\begin{corollary}\label{cor:map-sep-genus}
Let $R$ be a ribbon graph with corresponding combinatorial map $(\sigma,\alpha,\varphi)$ and let $g>0$.  Then $R\in \mathcal{R}_g$ if and only if $(\varphi,\alpha,\sigma)$ is bipartite, $c(\sigma)$ has 
no 2-cycles, and 
\[g=\frac{-c(\sigma)+c(\alpha)+c(\varphi)}{2} -1.\]
\end{corollary}
The condition that $R$ has a bipartite dual isn't immediately any more useful than the condition that $R$ be properly 2-face-colorable, but we will see it's utility later.

\begin{proof}
The elements of $\mathcal{R}_g$ are all connected ribbon graphs, so we know that $R$ does in fact have an associated combinatorial map.  Since we assume $g>0$, we know that $R$ is not the annulus and thus 
we can assume it has no degree 2 vertices, hence $c(\sigma)$ has no 2-cycles.  The condition that $(\varphi,\alpha,\sigma)$ is bipartite comes from the equivalence between being able to properly 2-face-color
 a combinatorial map and its dual being bipartite.  Finally, we obtain the equality for $g$ by substituting $F= 2+g-g_R$ from Corollary \ref{cor:facecount} into the the formula for $g_R$ from 
Definition \ref{def:ribbon}, simplifying and `translating' via Table \ref{tab:gloss1}.
\end{proof}

Now we have precise conditions on whether a combinatorial map $(\sigma,\alpha,\varphi)$ corresponds to an element of $\mathcal{R}_g$.  If we knew the possible numbers of edges for an element of 
$\mathcal{R}_g$, we could determine which symmetric groups to search (recall that $\sigma,\alpha,\varphi$ are all in $S_{2n}$ where $n$ is the number of edges of the underlying graph), and enumerate 
(albeit quite inefficiently) the elements of $\mathcal{R}_g$.  Now we'll determine the possible numbers of edges for an element of $\mathcal{R}_g$ before using the bipartiteness of the duals to 
more efficiently find the elements of $\mathcal{R}_g$.

\begin{lemma}\label{lem:edge-bound}
If $(\sigma,\alpha,\varphi) \in \mathcal{R}_g$ then for all $g>0$.
\[1+g\le c(\alpha) \le 4g\]
\end{lemma}

\begin{proof}
Recall that for $g>0$, we can assume there are no degree 2 vertices in an element of $\mathcal{R}_g$.  Additionally, since each edge is incident to both colors of faces in the surface, 
the degree of each vertex must be even.
Thus every vertex has degree at least 4 and we get $V\le \frac{E}{2}$ or $c(\sigma)\leq \frac{c(\alpha)}{2}$. By separating
there must at least be two faces, $c(\varphi) \geq 2$. So substituting into Corollary \ref{cor:map-sep-genus}
gives the right inequality.

From Corollary \ref{cor:map-sep-genus} we obtain that
\begin{equation*}
c(\alpha)=2g-c(\varphi)+c(\sigma)+2 \,.
\end{equation*}
By Corollary \ref{cor:facecount}, $c(\varphi) = 2+ g-g_R$ (since $c(\varphi)$ is the number of faces of $R$). Substituting this for $c(\varphi)$ gives 
\begin{equation*}
c(\alpha)= g+g_R+c(\sigma) \,.
\end{equation*}
Since $g_R \geq 0$ and $c(\sigma)\geq 1$, the left inequality follows.
\end{proof}

Now we introduce a generalization of combinatorial maps which will allow us to more quickly enumerate elements of $\mathcal{R}_g$ using a computer.

\begin{defn}\label{def:hypermap}
A hypermap is a triple of permutations $(\sigma,\alpha,\varphi)$ in $S_{n}$ such that\\
i) $\langle \sigma,\alpha,\varphi\rangle$ acts transitively on $\{1,\ldots, n\}$\\
ii) $\varphi\circ\alpha\circ\sigma = \;\text{Id}$.
\end{defn}
Note that this definition is effectively a relaxation of Definition \ref{def:combinatorialmap}, since it no longer requires that one of the permutations be a fixed-point-free involution.  Hypermaps 
generalize combinatorial maps analogously to how hypergraphs generalize graphs, so it's conventional to refer to the cycles of $\sigma$ as hypervertices, the cycles of $\alpha$ as hyperedges, the cycles 
of $\varphi$ as hyperfaces, and the labels $1,2,\ldots, n$ as brins (or as darts, bits, or blades).  

Isomorphisms of hypermaps are defined analogously to isomorphisms for combinatorial maps:

\begin{defn}\label{def:hypermap-iso}An isomorphism between two hypermaps $(\sigma_1,\alpha_1,\varphi_1)$ and $(\sigma_2,\alpha_2,\varphi_2)$
in $S_n$, is a permuation $\gamma \in S_n$ with
\[\gamma^{-1}\sigma_1\gamma = \sigma_2,\quad\gamma^{-1}\alpha_1\gamma = \alpha_2,\quad\gamma^{-1}\varphi_1\gamma = \varphi_2.\]
\end{defn}
Note that as an immediate consequence, any automorphism $\gamma$ of a hypermap $(\sigma,\alpha,\varphi)$ must commute with each of $\sigma,\alpha,$ and $\varphi$.

The following theorem due to Walsh gives the connection between hypermaps and the elements of $\mathcal{R}_g$ (language is slightly modified since what Walsh calls a hypermap we call an 
isomorphism class of hypermap).
\begin{theorem}\cite{walsh}
There is a one-to-one correspondence isomorphism classes of hypermaps and isomorphism class of 2-vertex-colored bipartite combinatorial maps.  This correspondence maps hypervertices to vertices of one 
color class, hyperedges to vertices of the other color class, hyperfaces to faces, and brins to edges.
\end{theorem}
We'll briefly illustrate this correspondence before moving forward, but the best illustration is a picture (see Figure \ref{fig:hypermap-ex}).  Given a 2-vertex-colored bipartite combinatorial map $(\sigma,\alpha,\varphi)$ with $n$ edges (so $\sigma,\alpha,\varphi \in S_{2n}$) 
we'll consider the cycles of $\sigma$ which correspond to black verices separately from the ones corresponding to white vertices.  Each edge of a bipartite graph is incident to exactly one vertex in 
each color class, so the cycles of $\sigma$ which correspond to black vertices induce a permutation on the edges of the graph (with one cycle per black vertex) and so the cycles of $\sigma$ corresponding to 
the white vertices (with the induced permutation having one cycle per white vertex).  Labelling the edges $1,2,\ldots, n$ we now have elements $\overline{\sigma}$ and $\overline{\alpha}$ in $S_n$ induced 
by the action of $\sigma$ restriced (respectively) to the black and white vertices.  For now we simply define $\overline{\varphi} := (\alpha\circ \sigma)^{-1}$, but from \cite{walsh} and \cite{lando} 
$\overline{\varphi}$ has the same number of cycles as $\varphi$, with each cycle of $\overline{\varphi}$ having exactly half the length of the corresponding cycle in $\varphi$.  

\begin{defn}
The \textit{genus} of a hypermap $(\sigma,\alpha,\varphi)$ is defined to be the genus $g_R$ of the ribbon graph $R$ underyling its corresponding 2-colored bipartite combinatorial map.  It satisfies
\[2-2g_R = c(\sigma)+c(\alpha)-n + c(\varphi).\]
\end{defn}

\begin{figure}
\begin{center}
\begin{tikzpicture}
\fill (0,3) circle [radius=8pt];
\draw (-3,0) circle [radius =8pt]; 
\draw (0,0) circle [radius=8pt];
\draw (3,0) circle [radius=8pt];
\draw (-.25,3) -- (-2.8,.2);
\draw (.25,3) -- (2.8,.2);
\draw (0,3) arc (140:212:2.4);
\draw (0,3) arc (40:-32.2:2.4);
\node at (-0.12,3.0) [label={left:\footnotesize{$1$}}]{};
\node at (-0.08, 2.5) [label={left:\footnotesize{$2$}}]{};
\node at (0.53, 2.5) [label={left:\footnotesize{$3$}}]{};
\node at (0.94, 2.5) [label={left:\footnotesize{$4$}}]{};
\node at (-2.56,-0.14) [label={above:\footnotesize{$5$}}]{};
\node at (-.15, 0.05) [label={above: \footnotesize{$6$}}]{};
\node at (.5, 0.035) [label={above:\footnotesize{$7$}}]{};
\node at (2.85, 0.06) [label={above:\footnotesize{$8$}}]{};
\end{tikzpicture}
\quad
\begin{tikzpicture}
\fill (0,3) circle [radius=8pt];
\draw (-3,0) circle [radius =8pt]; 
\draw (0,0) circle [radius=8pt];
\draw (3,0) circle [radius=8pt];
\draw (-.25,3) -- (-2.8,.2);
\draw (.25,3) -- (2.8,.2);
\draw (0,3) arc (140:212:2.4);
\draw (0,3) arc (40:-32.2:2.4);
\node at (146:2.2) [label={above:\footnotesize{$1$}}]{};
\node at (100:1.7) [label={left:\footnotesize{$2$}}]{};
\node at (64:1.85) [label={left:\footnotesize{$3$}}]{};
\node at (45:2.3) [label={left:\footnotesize{$4$}}]{};
\end{tikzpicture}
\caption{On the left is a drawing of the combinatorial map $\sigma= (1,2,3,4)(6,7), \alpha = (1,5)(2,6)(3,7)(4,8), \varphi = (1,8,4,7,2,5)(3,6)$, and on the right is the corresponding
hypermap with $\sigma =(1,2,3,4), \alpha = (2,3), \varphi = (1,4,2)$.  Each drawing can be viewed as an embedding in the sphere.}\label{fig:hypermap-ex}
\end{center}\end{figure}
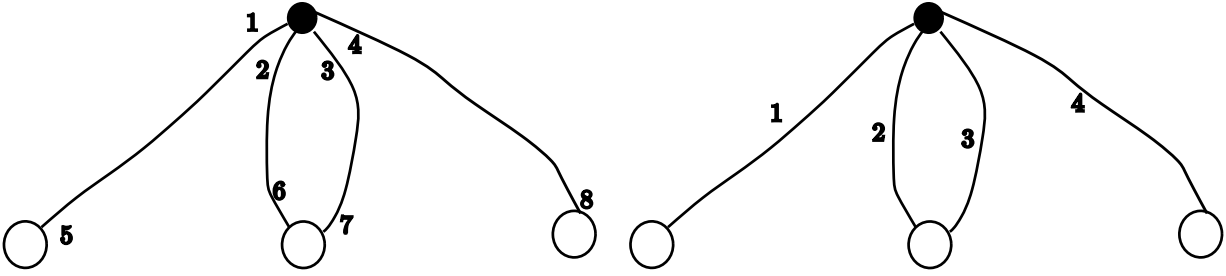

We know that each element of $\mathcal{R}_g$ has a bipartite dual map, and this is where we relate our problem to hypermap enumeration.  Although there is a one-to-one correspondence between 
isomorphism class of hypermap and isomorphism classes of 2-vertex-colored bipartite combinatorial map, there are twice as many labels in a combinatorial map as in the corresponding hypermap, which 
results in vastly more possible labellings.  This in turn translates to a much larger search space for our computer program.  Table \ref{tab:gloss2} extends our previous glossary between ribbon graphs and 
combinatorial maps to one between properly 2-face-colored ribbon graphs, their corresponding 2-face-colored combinatorial map, and the hypermap corresponding to their dual.

\begin{table}
\begin{center}
\begin{tabular}{|c|c|c|}
Ribbon Graph & Combinatorial Map& Hypermap of Dual Map\\
\hline
Number of Vertices & $c(\sigma)$ & $c(\overline{\varphi})$\\
\hline
Number of Edges & $c(\alpha)$ & $n$\\
\hline
Number of Faces & $c(\varphi)$ & $c(\overline{\sigma})+c(\overline{\alpha})$\\
\end{tabular}\caption{A glossary for translating between properly 2-face-colored ribbon graphs, their corresponding combinatorial map $(\sigma,\alpha,\varphi)$, and the hypermap $(\overline{\sigma},
\overline{\alpha},\overline{\varphi})$ corresponding to the 2-vertex-colored dual of the combinatorial map.  Here $n$ is the number of brins.}\label{tab:gloss2}
\end{center}
\end{table}

Translating Corollary \ref{cor:map-sep-genus} into terms of hypermaps we have
\begin{theorem}\label{thm:hypermap-condition}
A combinatorial map $(\sigma,\alpha,\varphi)$ is an element of $\mathcal{R}_g$ if and only if its dual map is bipartite and (after choice of 2-vertex-coloring) corresponds to a hypermap 
$(\overline{\sigma},\overline{\alpha},\overline{\varphi})$ such that $\overline{\varphi}$ has no fixed points and satisfying
\[g = \frac{-c(\overline{\varphi})+n+ c(\overline{\sigma})+c(\overline{\alpha})}{2}-1\]
where $n$ is the number of edges of the combinatorial map.
\end{theorem}

\begin{proof}
From the construction of the hypermap corresponding to a 2-vertex-colored bipartite map, we have $c(\varphi) = c(\overline{\sigma})+c(\overline{\alpha})$ (total vertices $=$ black vertices $+$ 
white vertices), $c(\alpha) =n$ (these are both the number of edges), and $c(\sigma) = c(\overline{\varphi})$ (these are both the number of faces in the dual).  Similarly, $\overline{\varphi}$ has a
fixed point (a 1-cycle) if and only if $\sigma$ has a 2-cycle (a degree 2 vertex).  Then the proof follows immediately from Corollary \ref{cor:map-sep-genus}.
\end{proof}

\section{Computation and Results}
Here we introduce an algorithm to enumerate the elements of $\mathcal{R}_g$.  Then we use the results of the previous section to compute the sizes of $|\mathcal{C}_g|, |\mathcal{L}_g|,$ 
and $|\mathcal{M}_g|$.

At this point we've bounded the number of edges by $g+1\le E \le 4g$ so we could search all hypermaps in $S_{E}$ to find the ones which are dual to an element of $\mathcal{R}_g$ as follows:

For each $E$ with $g+1\le E\le 4g$, we construct list of triples $(S,A,F)$ of possible cycle types for a hypermap $(\sigma,\alpha,\varphi)$ satisfying 
Theorem \ref{thm:hypermap-condition}.  In the correspondence between 2-colored maps and hypermaps, $\sigma$ and $\alpha$ correspond to the two colors of the vertices. The actual choice of coloring doesn't 
matter to us, we always pick $S$ and $A$ so that $\alpha$ has no fewer cycles than $\sigma$.  In all the most computationally demanding cases, this results in there being more permutations with cycle 
type $S$ than with either $A$ or $F$.  The assignment of labels is free, so for each triple we fix a permutation of cycle type $S$ to be $\sigma$.  Since $\varphi\circ\alpha\circ\sigma= 1$, the 
choice of any two elements of $(\sigma,\alpha,\varphi)$ determines the third uniquely.  Then from $A$ and $F$, we select whichever cycle type corresponds to the smallest number of permutations.
Without loss of generality, assume that it is $A$.  The for each permutation $\alpha$ with cycle-type $A$, we consider the hypermap $(\sigma,\alpha,(\alpha\circ\sigma)^{-1})$ and check to see if it satisfies 
Theorem \ref{thm:hypermap-condition}.  The actual search through all permutations with a given cycle-type is done in parallel using an iterative procedure based on several efficient algorithms for 
generating permutations and combinations found in \cite{ruskey}.

We're only interested in isomorphism classes of hypermaps, so whenever we find a $\alpha$ with $(\sigma,\alpha, (\alpha\circ\sigma)^{-1})$ satisfying Theorem \ref{thm:hypermap-condition} we check if there 
is a $\rho \in S_n$ such that $\rho^{-1}\sigma\rho=\sigma$ and $\rho^{-1}\alpha\rho$ comes before $\alpha$ in lexicographic order and only store $(\sigma,\alpha, (\alpha\circ\sigma)^{-1})$ as the 
representative for an element of $\mathcal{R}_g$ if $\alpha$ comes first in lexicographic under order among all such conjugations.

This would generate a complete list of isomorphism classes of hypermaps corresponding to the elements of $\mathcal{R}_g$.  Unfortunately, for the purposes of finding
$|\mathcal{R}_g|$, we've overcounted, since Walsh's correspondence is between \textit{2-colored} bipartite maps and hypermaps, and the choice of how to assign the colors means that most elements
of $\mathcal{R}_g$ correspond to two distinct hypermaps.  To avoid extensive isomorphism testing to resolve this issue, we instead restrict our search to cases where the cycle type of $\sigma$ does not 
come after that of $\alpha$ in lexicographic order (when written as partitions).  When $\sigma$ and $\alpha$ have the same cycle type we choose the first representation in lexicographic order of the two 
hypermaps $(\sigma,\alpha,(\alpha\circ\sigma)^{-1})$ and $(\alpha,\sigma, (\sigma\circ\alpha)^{-1})$ as a canonical form for the union or their isomorphism classes and only store such a hypermap if it 
is already in canonical form.  With this convention the algorithm produces a list with a single representative for the dual of each element in $\mathcal{R}_g$.


From Lemma \ref{lem:rgcg}, we can similarly determine the elements of $\mathcal{C}_g$ if we know the elements of
$\mathcal{R}_{\hat{g}}$ for all $\hat{g}\le g$.  Once $\mathcal{R}_{\hat{g}}$ has been computed for all $\hat{g}\le g$, we can inductively find the elements of $\mathcal{C}_g$ by taking the underlying graphs of
each map in $\mathcal{R}_g$ and graph isomorphism testing them against the elements of $\mathcal{C}_{\hat{g}}$ for $\hat{g}\le g$. 

A computer program to compute the values of $\mathcal{R}_g$ and $\mathcal{C}_g$ using this method can be found at \href{https://github.com/caagaard/Minseps}{this GitHub link}.  The time and memory 
needs of such a program are quite large, and required preallocating memory to avoid overflows.  To do this we use the following formula due to Frobenius:
\begin{theorem}\cite{lorenz}\label{thm:frob} Let $G$ be a finite group and let $C_1,\ldots, C_k$ be arbitrary conjugacy classes in $G$.  Define:
\[N(C_1,\ldots, C_k) := \lvert \{(g_1,\ldots, g_k) \in C_1\times\ldots\times C_k : g_1g_2\ldots g_k = \text{Id}\}\rvert.\]
Then
\[N(C_1,\ldots, C_k) = \frac{|C_1| \ldots |C_k|}{|G|}\sum_{V_i \in \text{Irr}\;G}\frac{\chi_i(C_1)\ldots\chi_i(C_k)}{\chi_i(\text{Id})^{k-2}}\]
where $\text{Irr}\;G$ is the set of irreducible representations of $G$ and $\chi_i$ is the character corresponding to the representation $V_i$.
\end{theorem}
Given a number of brins $n$ and a triple of cycle types $(S,A,F)$ for elements of $S_n$, this formula doesn't require that the counted triples generate transitive permutation groups (a requirement for a 
 hypermap in Definition \ref{def:hypermap}) or identify isomorphic hypermaps (different labellings of the same embedding) so it gives a 
significant overestimate.  In all the cases where memory is a concern, $S$ is the conjugacy class of $n$-cycles, each of which acts transitively on $\{1,\ldots, n\}$, so there actually isn't any 
overestimating related to the transitivity condition in practice.  Since we fix an element of the first conjugacy class, $S$, in our algorithm, we divide the formula for $N(S,A,F)$ by $|S|$, which is 
closer to the number of hypermaps that must be stored, but still doesn't identify identify isomorphic hypermaps with the same $\sigma$ (when the isomorphism $\rho$ commutes with $\sigma$).  
With enough memory this upper bound would be sufficient for preallocation, but using this upper bound would 
preallocate more memory than the highest memory node on our cluster has for some cases.  In all cases where memory usage was a concern $S$ is the conjugacy class of $n$-cycles in $S_n$, so we further 
divide the formula for $N(S,A,F)$ by $n$ (the number of elements in $S_n$ that commute with an arbitrary $n$-cycle).  Outside of the few cases where it is exact, this is just an estimate of the number of 
hypermaps that must be stored (it would be a slight underestimate if we needed a representative of each isomorphism class, but we equate isomorphism classes that differ only by changing the 2-coloring), 
but it turns out to be a close enough estimate that the arrays storing output never needed to be resized more than once, which could fit in our memory.   The GitHub repository 
linked above has suggestions on how to modify the code for users who cannot afford even this limited array resizing.   

With the values of the $|\mathcal{R}_g|$ and $|\mathcal{C}_g|$ computed, we apply Corollary \ref{cor:lgcount} to obtain the values of the $|\mathcal{L}_g|$.  Then we inductively compute the $|\mathcal{M}_g|$
 using $|\mathcal{M}_0| = 1$ from \cite{bernhard} and $|\mathcal{M}_g|=|\mathcal{M}_{g-1}|+|\mathcal{L}_g|$ from Definition \ref{def:graphsets}.

The results of running this algorithm for $g\le 5$ are shown in Table \ref{tab:results}
\begin{table}\begin{center}\begin{tabular}{c|c|c|c|c}
$g$&$\lvert\mathcal{R}_g\rvert$&$\lvert\mathcal{C}_g\rvert$&$\lvert\mathcal{L}_g\rvert$&$\lvert\mathcal{M}_g\rvert$\\
\hline
0 & 1 & 1 & 1 & 1\\
\hline
1 & 3 & 3 & 4 & 5\\
\hline
2 & 31& 17 &  21& 26 \\
\hline
3 & 1831 & 164 & 191 & 217 \\
\hline
4 & 462638 & 3096 &  3338 & 3555\\
\hline
5 & 243066565 & 111445 & 115438 & 118993
\end{tabular}\caption{Sizes of $\mathcal{R}_g,\mathcal{C}_g,\mathcal{L}_g,$ and $\mathcal{M}_g$ for $g\le 5$.}\label{tab:results}\end{center}\end{table}

\section{Concluding Remarks}
A search of the OEIS does not find any sequences matching any of the sequences $\mathcal{R}_g$, $\mathcal{C}_g, \mathcal{L}_g,$ or $\mathcal{M}_g$.  It thus appears unlikely that any currently
known formula will yield correct values for the cardinalities of these sets.

At this point, we do not see any way forward towards a method to directly determine the values of $|\mathcal{C}_g|,|\mathcal{L}_g|$, or $|\mathcal{M}_g|$ without the computationally intensive process of
directly computing $\mathcal{R}_g$ and then determining $|\mathcal{C}_g|$ through graph isomorphism testing. It would be interesting to attempt to find a recursive formula for the values of 
$|\mathcal{R}_g|$ using some of the techniques for enumerating isomorphism
classes of similar flavors of nonplanar hypermaps and combinatorial maps developed in \cite{mednykh_enumeration_2006} and \cite{mednykh_enumeration_2010}.  A rough idea of the method developed there is to 
use the fact that it's easier to counted rooted maps and hypermaps (maps/hypermaps with a distinguished edge) directly than to count unrooted ones, 
and the authors develop formulae to convert counts of rooted maps/hypermaps into counts of unrooted maps and hypermaps.  This is done by considering rooted objects which can be realized as quotients under 
an automorphism of the unrooted objects in question.  For formal definitions and more details on quotients of maps see \cite{jones} or \cite{liskovets} or the unpublished \cite{nedela}.

There is an important distinction between the sets considered in those papers and the $\mathcal{R}_g$ we consider.  The first is that the families of maps and hypermaps considered in those 
papers are closed under quotients by map automorphisms. An example of the set of 
hypermaps considered in \cite{mednykh_enumeration_2010} may help to clarify this.  The authors enumerate the number of hypermaps of genus $g$ with $n$ edges.  If $M$ is a hypermap of with genus $g$ with $n$ edges, and $f$ is an automorphism of $M$, then the quotient $M/f$ is a hypermap with genus at most $g$ and at most $n$ edges (strictly fewer unless $f$ is the
identity).  So, if ones proceeds inductively with respect to genus and number of edges, they will have already considered these maps.  As observed in \cite{liskovets} this property is extremely useful for
such enumeration problems.  Without this property one needs to determine precisely which rooted objects arise as quotients of the unrooted ones being counted and then enumerate those.
The difficulty with using those techniques to enumerates elements of $\mathcal{R}_g$ arises here. The quotient of a map in $\mathcal{R}_g$ by an automorphism is not generally another map in $\mathcal{R}_g$, or even a map which be embedded as a minimal separating set in any surface.  As an example, consider the map and its quotient shown in Figure \ref{fig:map-aut} (for details on quotients of maps see \cite{liskovets} or \cite{nedela}).

\begin{figure}
\begin{center}
\begin{tikzpicture}[scale=0.8]
\draw (0,0) circle [radius=3];
\draw (2,0) circle [radius=1];
\draw (-4,0) circle [radius=1];
\fill (-3,0) circle [radius=2pt];
\fill (3,0) circle [radius=2pt];
\node at (174:3.2) [label={above:\footnotesize{$1$}}]{};
\node at (177:2.7) [label={above:\footnotesize{$6$}}]{};
\node at (179:3.3) [label={below:\footnotesize{$2$}}]{};
\node at (185:3.05) [label={below:\footnotesize{$3$}}]{};
\node at (2:3.2) [label={above:\footnotesize{$5$}}]{};
\node at (8.5:2.7) [label={above:\footnotesize{$7$}}]{};
\node at (1:2.7) [label={below: \footnotesize{$8$}}]{};
\node at (-6:2.8) [label={below:\footnotesize{$4$}}]{};
\end{tikzpicture}
\quad
\begin{tikzpicture}[scale=0.8]
\draw (-1,0) circle [radius=1];
\draw (0,0) -- (3,3);
\draw (0,0) -- (3,-3);
\fill (0,0) circle[radius=2pt];
\node at (125:.5) [label={above:\footnotesize{$1$}}]{};
\node at (187:.45) [label={below:\footnotesize{$2$}}]{};
\node at (322:1.1) [label={left:\footnotesize{$3$}}]{};
\node at (22:1.3) [label={left:\footnotesize{$6$}}]{};
\end{tikzpicture}
\caption{The combinatorial map $\sigma = (1,2,3,6)(4,5,7,8), \alpha = (1,2)(3,4)(5,6)(7,8)$ (left) and its quotient (right) under the automorphism $\rho = (1,7)(2,8)(3,4)(5,6)$.  The unlabelled edge ends
 that occur are called \textit{singular edges} in \cite{liskovets} and these would be fixed points under the involution $\alpha$ describing the quotient, which means it violates our definition of 
 combinatorial map.  Weakening the condition from \ref{def:combinatorialmap} that $\alpha$ be a fixed-point free involution, to merely requiring $\alpha$ is an involution (this is the definition used in 
 \cite{mednykh_enumeration_2006}) makes the counting problem more tractable by allowing maps with singular edges but in our setting requires determining what maps with singular edges can arise as 
 quotients of minimal separating sets.}\label{fig:map-aut}
\end{center}\end{figure}
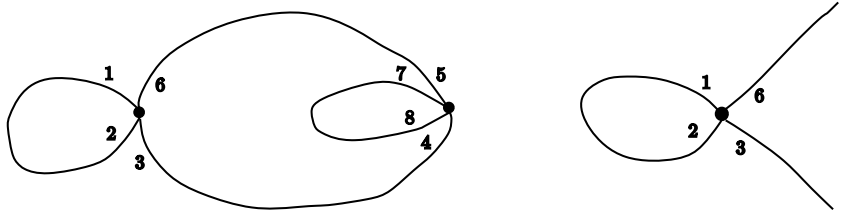

The map on the left is an element of $\mathcal{R}_2$, but the quotient map cannot be minimal separating in any genus.  Additionally, such a method would require a formula enumerating the
numer of rooted versions of the relevant combinatorial maps or hypermaps (which is as yet unknown).  To the best of our knowledge, in all existing work enumerating unrooted maps and hypermaps 
using the tools developed in \cite{mednykh_enumeration_2006}, \cite{mednykh_enumeration_2010}, and \cite{liskovets} the familes of objects being considered tend to either count all maps/hypermaps 
with a given genus and edge count or to heavily restrict to counting only regular maps/hypermaps (\cite{krasko3regular},\cite{krasko2019enumeration},\cite{krasko2019enumeration2}).  The case of 
all maps/hypermaps puts no restrictions on cycle types of the permutations, while a requiring a map to be $k$-regular for some $k$
fully determines the cycle types of $\sigma$ and $\alpha$.  In our case, there are some restrictions on cycle types, but nothing nearly as strong as regularity.  From some initial investigations we've done, 
considering the hypermaps associated to elements of $\mathcal{R}_g$ seems somewhat ammenable to the application of these methods, so they may form an interesting class of maps or hypermaps for 
further study of these methods.

\nocite{OSCAR}
\nocite{Julia-2017}
\bibliographystyle{plain}
\bibliography{classification_of_minseps}
\end{document}